\documentclass[12pt]{amsart}
\usepackage{amssymb}
\usepackage[all]{xy}

\input xy

\xyoption{all}

\newtheorem{lemma}{Lemma}[section]
\newtheorem{corollary}[lemma]{Corollary}
\newtheorem{theorem}[lemma]{Theorem}
\newtheorem{proposition}[lemma]{Proposition}

\theoremstyle{definition}
\newtheorem{remark}[lemma]{Remark}
\newtheorem{definition}[lemma]{Definition}

\newtheorem{example}[lemma]{Example}

\begin{document}

\title{Algebraic Cuntz-Krieger algebras}\footnotetext{This research was in part supported by a grant from IPM (No. 95170419).}

\author{Alireza Nasr-Isfahani}
\address{Department of Mathematics, University of Isfahan, P.O. Box: 81746-73441, Isfahan, Iran\\
  and School of Mathematics, Institute for Research in Fundamental Sciences (IPM), P.O. Box: 19395-5746, Tehran, Iran}
\email{nasr$_{-}$a@sci.ui.ac.ir / nasr@ipm.ir}



\begin{abstract} We show that a directed graph $E$ is a finite graph with no sinks if and only if for each commutative unital ring $R$, the Leavitt path algebra $L_R(E)$ is isomorphic to an algebraic Cuntz-Krieger algebra if and only if the $C^*$-algebra $C^*(E)$ is unital and $rank(K_0(C^*(E)))=rank(K_1(C^*(E)))$. Let $k$ be a field and $k^{\times}$ be a group of units of $k$. When $rank(k^{\times})< \infty$, we show that the Leavitt path algebra $L_k(E)$ is isomorphic to an algebraic Cuntz-Krieger algebra if and only if $L_k(E)$ is unital and $rank(K_1(L_k(E)))=(rank(k^{\times})+1)rank(K_0(L_k(E)))$. We also show that any unital $k$-algebra which is Morita equivalent or stably isomorphic to an algebraic Cuntz-Krieger algebra, is isomorphic to an algebraic Cuntz-Krieger algebra. As a consequence, corners of algebraic Cuntz-Krieger algebras are algebraic Cuntz-Krieger algebras.
\end{abstract}

\maketitle


\section{Introduction}

Cuntz-Krieger algebras, introduced and first investigated by Cuntz and Krieger \cite{CK} in 1980, is a prominent class of $C^*$-algebras arising from dynamical systems. The Cuntz-Krieger algebra $\mathcal{O}_A$ was originally associated to a finite square $\{0, 1\}$-matrix $A$ \cite{CK}, but it can also be viewed as the graph $C^*$-algebra of a finite directed graph with no sinks and no sources \cite{W}. Graph $C^*$-algebras and their generalizations have been intensively investigated by analysts for more than two decades (see \cite{R} for an overview of the subject).

The algebraic Cuntz-Krieger algebras arose as specific examples of fractional skew monoid rings \cite{AGGP}.
Leavitt path algebras are the algebraic version of graph $C^*$-algebras. Leavitt path algebras $L_k(E)$ are quotients of path algebras associated to an extended graph $\widehat{E}$ and a field $k$, modulo additional relations. An algebraic Cuntz-Krieger algebra $\mathcal{CK}_k(E)$ is a Leavitt path algebra $L_k(E)$ of a finite graph $E$ with no sinks and no sources.

Initially, Leavitt path algebras were introduced by P. Ara, M. A. Moreno
and E. Pardo in \cite{AMP} and by G. Abrams and G. Aranda Pino in \cite{AA1}. M. Tomforde in \cite{T} generalizes the construction of Leavitt path algebras by replacing the field $k$ with a commutative unital ring $R$. Leavitt path algebras are also a generalization of the algebras
constructed by Leavitt in \cite{Leavitt} to produce rings without the Invariant Basis Number property (i.e., $R_R^m\cong
R_R^n$ as left $R$-modules with $m\neq n$). Leavitt path algebras include many well-known algebras such as matrix algebras ${\mathbb M}_n(R)$ for $n\geq 1$, the Laurent polynomial ring $R[x,x^{-1}]$, or the Leavitt algebras $L(1,n)$ for $n\ge 2$.

Many times in the literature, the conditions characterizing some analytic properties of the graph $C^*$-algebra turned out to be exactly the same conditions characterizing the corresponding algebraic version of these properties. In this sense the Leavitt path algebra theory, being more recent, has benefited from the inspiration that the graph $C^*$-algebra world provided. This is the case once more for the topic discussed in the current paper: the analytic results were given in \cite{AR} for the Cuntz-Krieger algebras, and we give here the algebraic analogue for algebraic Cuntz-Krieger algebras. Some of the ideas in this paper are contained in \cite{AR}. Theorem 3.12 of \cite{AR} says that, the graph $C^*$-algebra $C^*(E)$ is isomorphic to a Cuntz-Krieger algebra if and only if $C^*(E)$ is unital and $rank(K_0(C^*(E)))=rank(K_1(C^*(E)))$. In algebraic setting we prove a slightly different theorem \ref{thm2}: "Let $k$ be a field such that $rank(k^{\times})< \infty$. Then the Leavitt path algebra $L_k(E)$ is isomorphic to an algebraic Cuntz-Krieger algebra if and only if $L_k(E)$ is unital and $rank(K_1(L_k(E)))=(rank(k^{\times})+1)rank(K_0(L_k(E)))$". We also show that the assumption $rank(k^{\times})< \infty$ is necessary.

The paper is organized as follows. In Section 2 we give all the background information,
definitions and basic properties of Leavitt path algebras that we need in this paper.

In Section 3 we give a characterization of algebraic Cuntz-Krieger algebras. In the first step of this process, we provide a class of operations on graphs that preserve isomorphism of associated Leavitt path algebras. With this useful result, we show that algebraic Cuntz-Krieger algebras are Leavitt path algebras of finite graphs with no sinks. Finally, in Corollary \ref{cor2} and Theorem \ref{thm2} we derive further conditions for $L_R(E)$ to be isomorphic to an algebraic Cuntz-Krieger $R$-algebras when $R$ is a commutative unital ring or a field.

In Section 4 we first show that the corner $P_XL_R(E)P_X$, where $X$ is a finite subset of $E^0$, is isomorphic to a Leavitt path algebra $L_R(E(T))$. Also we show that if $L_R(E)$ is an algebraic Cuntz-Krieger algebra, then $L_R(E(T))$ is an algebraic Cuntz-Krieger algebra. Finally after proving similar results for $\mathbb{M}_n(L_R(E))$ and $\mathbb{M}_\infty(L_R(E))$, we show that if $A$ is an algebraic Cuntz-Krieger $R$-algebra then for each positive integer $n$, $\mathbb{M}_n(A)$ is isomorphic to an algebraic Cuntz-Krieger algebra.

 In the last section, we show that if a unital $k$-algebra $A$ is Morita equivalent or stably isomorphic to an algebraic Cuntz-Krieger algebra, then it is isomorphic to an algebraic Cuntz-Krieger algebra. As a consequence, $A$ is an algebraic Cuntz-Krieger algebra if and only if the full $n\times n$ matrix algebra over $A$ is an algebraic Cuntz-Krieger algebra. Also we show that if $A$ is an algebraic Cuntz-Krieger algebra, then the corners $eAe$ and $e'\mathbb{M}_\infty(A)e'$ are isomorphic to algebraic Cuntz-Krieger algebras.
\section{Preliminaries}

A directed {\it graph} $E=(E^{0},E^{1},r_E,s_E)$ consists of two sets $E^{0}$ and $
E^{1}$ together with maps $r_E, s_E:E^{1}\rightarrow E^{0}$, identifying the range and source of each edge.  The elements of $
E^{0}$ are called \textit{vertices} and the elements of $E^{1}$ \textit{edges}.

If a vertex $v$ emits no edges, that is, if $s_E^{-1}(v)$ is empty, then $v$
is called a\textit{\ sink}. A vertex $v$ is called a \textit{regular vertex} if $s_E^{-1}(v)$ is a finite non-empty set. The set of regular vertices is denoted by $E^0_{\text{reg}}$. We let $E^0_{\text{sing}}:=E^0\backslash E^0_{\text{reg}}$ and refer to an element of $E^0_{\text{sing}}$ as a singular vertex.

A finite path $\mu $ in a
graph $E$ is a finite sequence of edges $\mu =e_{1}\dots e_{n}$ such that $
r_E(e_{i})=s_E(e_{i+1})$ for $i=1,\dots ,n-1$. In this case, $n=l(\mu)$ is the length
of $\mu $. We view the elements of $E^{0}$ as paths of length $0$. For any $n\in {\mathbb N}$ the set of paths of length $n$ is denoted by $E^n$. Also, $\text{Path}(E)$ stands for the set of all finite paths, i.e., $\text{Path}(E)=\bigcup_{n=0}^{\infty} E^n$. We denote
by $\mu ^{0}$ the set of the vertices of the path $\mu $, that is, the set $\{s(e_{1}),r(e_{1}),\dots ,r(e_{n})\}$.

A path $\mu $ $=e_{1}\dots e_{n}$ is \textit{closed} if $r(e_{n})=s(e_{1})$,
in which case $\mu $ is said to be {\it based at the vertex} $s(e_{1})$. The closed path $\mu $ is called a \textit{cycle} if it does not pass through any of its vertices twice, that is, if $s(e_{i})\neq s(e_{j})$ for every $i\neq j$. A cycle of length one is called a \textit{loop}.  An \textit{exit }for a path $\mu =e_{1}\dots e_{n}$ is an edge $e$ such that $s_E(e)=s_E(e_{i})$ for some $i$ and $e\neq e_{i}$.

A right-infinite path $\mu $ in a
graph $E$ is an infinite sequence of edges $\mu =e_{1}e_2e_3\dots$ such that $
r_E(e_{i})=s_E(e_{i+1})$ for each $i$. A left-infinite path $\mu $ in a
graph $E$ is an infinite sequence of edges $\mu =\dots e_{-3}e_{-2}e_{-1}$ such that $
r_E(e_{i})=s_E(e_{i+1})$ for each $i$. A bi-infinite path $\mu $ in a
graph $E$ is an infinite sequence of edges $\mu =\dots e_{-3}e_{-2}e_{-1}e_0e_{1}e_2e_3\dots$ such that $
r_E(e_{i})=s_E(e_{i+1})$ for each $i$. We denote by $E^\infty$ the set of all (right-, left-, bi-) infinite paths in $E$.

A path $\mu=e_1e_2e_3\cdots$ is called vertex-simple if the sequence $s(e_1), r(e_1), r(e_2), \cdots$ contains no repeated vertices. A graph $E$ is called path-finite if $E^\infty$ contains no vertex-simple paths. A graph $E$ is called row-finite if for each $v\in E^0$, $s_E^{-1}(v)$ is a finite set.

For each $e\in E^{1}$, we call $e^{\ast }$ a {\it ghost edge}. We let $r_E(e^{\ast}) $ denote $s_E(e)$, and we let $s_E(e^{\ast })$ denote $r_E(e)$.

\begin{definition}\label{GA} {\rm Let $E$ be a graph. The \textit{graph $C^*$-algebra} $C^*(E)$ is the universal $C^*$-algebra generated by mutually orthogonal projections $\{p_v:v\in E^{0}\}$ together with partial isometries with mutually orthogonal ranges $\{s_e:e\in E^{1}\}$ which satisfy the following
conditions:

(1) (The ``CK-1 relations") For all $e\in E^{1}$, $s_e^*s_e=p_{r(e)}$ and $s_es_e^{\ast}\leq p_{s(e)}$;

(2) (The ``CK-2 relations") For every regular vertex $v\in E^{0}$,
\begin{equation*}
p_v=\sum_{\{e\in E^{1}|\ s_E(e)=v\}}s_es_e^{\ast}.
\end{equation*}}
\end{definition}

\begin{definition}\label{Lpa} {\rm Let $E$ be an arbitrary graph and $R$ be a commutative ring with unit. The \textit{Leavitt path algebra $L_{R}(E)$ with coefficients in $R$} is the universal $R$-algebra generated by a set $\{v:v\in E^{0}\}$ of pairwise orthogonal idempotents together with a set of variables $\{e,e^{\ast }:e\in E^{1}\}$ which satisfy the following
conditions:

(1) $s_E(e)e=e=er_E(e)$ for all $e\in E^{1}$;

(2) $r_E(e)e^{\ast }=e^{\ast }=e^{\ast }s_E(e)$\ for all $e\in E^{1}$;

(3) (The ``CK-1 relations") For all $e,f\in E^{1}$, $e^{\ast}e=r_E(e)$ and $
e^{\ast}f=0$ if $e\neq f$;

(4) (The ``CK-2 relations") For every regular vertex $v\in E^{0}$,
\begin{equation*}
v=\sum_{\{e\in E^{1}|\ s_E(e)=v\}}ee^{\ast}.
\end{equation*}}
\end{definition}

Another definition for $L_R(E)$ can be given using the extended graph $\widehat{E}$. This graph has the same set of vertices $E^0$ and the same set of edges $E^1$ together with the so-called ghost edges $e^*$ for each $e\in E^1$, whose directions are opposite to those of the corresponding $e\in E^1$. $L_R(E)$ can be defined as the usual path algebra $R\widehat{E}$ with coefficients in $R$ subject to the Cuntz-Krieger relations (3) and (4) above.

\begin{definition}\label{CuntzKriegerEsystem} {\rm Let $E$ be a graph and $\mathcal{A}$ be an $R$-algebra with involution $\ast$. A \textit{Cuntz-Krieger $E$-family in $\mathcal{A}$} is a collection $\Sigma = (S_{\mu})_{\mu \in E^0 \cup E^1} \subset \mathcal{A}$ which satisfies the following relations:

(1) For all $v, w \in E^0$, $S_vS_v = S_v$ and $S_vS_w = 0$ if $v\neq w$.

(2) $S_v^{\ast} = S_v$ for all $v \in E^0$.

(3) $S_{s(e)}S_{e} = S_{e} = S_{e}S_{r(e)}$ for all $e\in E^{1}$.

(4) For all $e,f\in E^{1}$, $S_e^{\ast}S_e=S_{r(e)}$ and $S_e^{\ast}S_f=0$ if $e\neq f$.

(5) For every regular vertex $v\in E^{0}$,
\begin{equation*}
S_v=\sum_{\{e\in E^{1},\ s(e)=v\}}S_eS_e^{\ast}.
\end{equation*}}
\end{definition}

Let $A$ be an $R$-algebra with a Cuntz-Krieger $E$-family, thus by the Universal Homomorphism Property of $L_R(E)$, there is a unique $R$-algebra homomorphism from $L_R(E)$ to $A$ mapping the
generators of $L_R(E)$ to their appropriate counterparts in $A$. We will refer to this property as the
Universal Homomorphism Property of $L_R(E)$.

\begin{definition}\label{CKa} {\rm Let $E$ be a finite graph with no sinks and no sources and $R$ a commutative ring with unit. The Leavitt path algebra $L_R(E)$ is called algebraic Cuntz-Krieger algebra, which is denoted by $\mathcal{CK}_R(E)$.
}
\end{definition}

If $E$ has a finite number of vertices, then $L_{R}(E)$ is unital with $\sum_{v\in E^0} v=1_{L_{R}(E)}$; otherwise, $L_{R}(E)$ is a ring with a set of local units (i.e., a set of elements $X$ such that for every finite collection $a_1,\dots,a_n\in L_R(E)$, there exists $x\in X$ such that $a_ix=a_i=xa_i$) consisting of sums of distinct vertices of the graph.

If $\mu = e_1 \dots e_n$ is a path in $E$, we write $\mu^*$ for the element $e_n^* \dots e_1^*$ of $L_{R}(E)$. With this notation it can be shown that the Leavitt path algebra $L_{R}(E)$ can be viewed as
$$L_{R}(E) = \text{span}_{R}\{\alpha\beta^*: \alpha,\beta \in \text{Path}(E) \text{ and } r(\alpha)=r(\beta)\}$$
and $r v \neq 0$ for all $v \in E^0$ and all $r \in R\setminus\{0\}$ (see \cite[Proposition 3.4]{T}). Also $L_R(E)$ is a $\ast$-algebra with linear anti-multiplicative involution defined by $(\sum_{i=1}^{n}r_i\alpha_i\beta_i^{\ast})^{\ast}=\sum_{i=1}^{n}r_i\beta_i\alpha_i^{\ast}$.

Let $G$ be a group. A ring $A=\oplus_{g\in G}A_g$ is called a $G$-graded ring, if each $A_g$ is an additive subgroup of $A$ and $A_gA_{g'}\subseteq A_{g+g'}$ for all $g, g'\in G$. A $G$-graded ring $A=\oplus_{g\in G}A_g$ is called a strongly graded ring if $A_gA_{g'}=A_{g+g'}$ for all $g, g'\in G$. Let $\Phi:A\rightarrow B$ be a ring homomorphism between $G$-graded rings. $\Phi$ is a graded ring homomorphism if $\Phi(A_g)\subseteq B_g$, for all $g\in G$.
Leavitt path algebras can be viewed as graded algebras. Let $G$ be a group with the identity element $e$ and $w:E^1\rightarrow G$ be a weight map. Also let $w(\alpha^\ast)=w(\alpha)^{-1}$ and $w(v)=e$, for each $\alpha\in E^1$ and $v\in E^0$. Thus the path algebra $R\widehat{E}$ of the extended graph $\widehat{E}$ is a $G$-graded $R$-algebra and since Cuntz-Krieger relations are homogeneous, $L_R(E)$ is a $G$-graded $R$-algebra. The natural grading given to a Leavitt path algebra is a $\mathbb{Z}$-grading by setting $w(\alpha)=1$,  $w(\alpha^\ast)=-1$ and $w(v)=0$, for each $\alpha\in E^1$ and $v\in E^0$. In this case the Leavitt path algebra can be decomposed as a direct sum of homogeneous components $L_R(E)=\bigoplus_{n\in {\mathbb Z}} L_R(E)_n$ satisfying $L_R(E)_nL_R(E)_m\subseteq L_R(E)_{n+m}$. Actually, $$L_R(E)_n=\text{span}_R\{pq^*: p,q\in \text{Path}(E) , l(p)-l(q)=n \}.$$ Every element $x\in L_R(E)_n$ is a homogeneous element of degree $n$.

An ideal $I$ is graded if it inherits the grading of $L_R(E)$, that is, if $I=\bigoplus_{n\in {\mathbb Z}} (I\cap L_R(E)_n)$.
Tomforde in \cite{T} (see also \cite[Theorem 3.5]{AMM}) proved that: (Graded Uniqueness Theorem) "Let $E$ be a graph and let $L_R(E)$
be the associated Leavitt path algebra with the usual $\mathbb Z$-grading. If $A$ is a $\mathbb Z$-graded ring, and
$\pi :L_R(E)\rightarrow A$ is a graded ring homomorphism with $\pi(rv)\neq 0$ for all $v\in E^0$ and $r\in R\backslash\{0\}$, then $\pi$ is injective."

We define a relation $\geq $ on $E^{0}$ by setting $v\geq w$ if there exists
a path $\mu$ in $E$ from $v$ to $w$, that is, $v=s(\mu)$ and $w=r(\mu)$. A subset $X$ of $E^{0}$ is called \textit{hereditary} if for each $v\in X$, $v\geq w$ implies that $w\in X$. For any subset $X\subseteq E^0$, the smallest hereditary subset of $E^0$ containing $X$ is denoted by $H_E(X)$. A subset $H \subseteq E^0$ is called \textit{saturated} if for any regular vertex $v$, $r(s^{-1}(v))\subseteq H$ implies that $v\in H$. An ideal $I$ of $L_R(E)$ is called basic if $rv\in I$ for $r\in R\setminus\{0\}$ implies that $v\in I$. Tomforde \cite[Theorem 7.9]{T} proved that the map $H\longrightarrow  I_H$ defines a lattice isomorphism between the saturated hereditary subsets of $E^0$ and the graded basic ideals of $L_R(E)$, where $I_H$ is a two-sided ideal in $L_R(E)$ generated by a saturated hereditary subset $H$ of $E^0$.

A right infinite path $\tau = e_1 e_2 \ldots$ in $E$ is called periodic, if there exist integers $j,k \geq 1$, such that $e_{n+k}= e_{n}$ for every $n \geq j$.
In this case, it is clear that the path $\rho = e_j \ldots e_{j+k-1}$ is closed. Take $j$ and $k$ such that $j+k$ is the smallest possible value which satisfies the condition $e_{n+k}= e_{n}$ for every $n \geq j$ and consider the paths $\alpha = e_1 \ldots e_{j-1}$ and $\lambda = e_j \ldots e_{j+k-1}$. The pair $(\alpha, \lambda)$ is called seed of $\tau$. Of course $\alpha$ may have zero length. In any case, $\lambda$ is a closed path, which is called the period of $\tau$. A right infinite path $\tau$ which is periodic and its period is a closed path without exits (which means that it has to be a cycle without exits), is called infinite discrete essentially aperiodic trail.
For any infinite discrete essentially aperiodic trail which is parameterized by the seed $(\alpha,\lambda_\alpha)$ of the trail (that is, $\alpha  \in \text{Path}(E)$ is its essential head and $r(\alpha)$ is visited by the cycle without exits $\lambda_{\alpha}$), the path $\alpha$ is called a distinguished path. In the case $l(\alpha) = 0$, $\alpha$ is called a distinguished vertex. For any distinguished path $\alpha$, $\alpha\lambda_{\alpha}\alpha^*$ is denoted by $\omega_\alpha$.

To finish this section we introduce a generalized uniqueness theorem \cite[Theorem 5.2]{GN}, which we will use later.

\begin{theorem}{\cite[Theorem 5.2]{GN}}\label{UniquenessThm} Let $E$ be a graph, $R$ be a commutative ring with unit and $\mathcal{A}$ be an $R$-algebra. Consider $\Phi: L_R(E) \rightarrow \mathcal{A}$ a ring homomorphism. Then the following conditions are equivalent:
\begin{itemize}
\item[(i)] $\Phi$ is injective;
\item[(ii)] the restriction of $\Phi$ to $M_R(E)$ is injective;
\item[(iii)] both these conditions are satisfied:
\begin{itemize}
\item[(a)] $\Phi(rv) \neq 0$, for all $v \in E^0$ and for all $r \in R \setminus \{0\}$;
\item[(b)] for every distinguished path $\alpha$ the $\ast$ $R$-algebra $< \Phi(\omega_\alpha) >$ generated by $\Phi(\omega_\alpha)$ is $\ast$-isomorphic to $R[x,x^{-1}]$; that is, $< \Phi(\omega_\alpha) > \ \cong R[x,x^{-1}]$.
\end{itemize}
\end{itemize}
\end{theorem}
\section{Characterization of algebraic Cuntz-Krieger algebras}

In this section we give a characterization of algebraic Cuntz-Krieger algebras.

\begin{definition}{\cite[Definition 3.6]{AR}}\label{def1} {\rm
Let $E$ be a graph, $H$ be a hereditary subset of $E^0$ and $F(H)=\{\alpha|\alpha=e_1e_2\cdots e_n\in \text{Path}(E), s_E(e_n)\not\in H, r_E(e_n)\in H \}$. Let $\overline{F}(H)$ be another copy of $F(H)$ and for each $\alpha\in F(H)$, the copy of $\alpha$ in $\overline{F}(H)$ is denoted by $\overline{\alpha}$. Define a graph $E(H)$ as follows:
$$E(H)^0=H\cup F(H),$$
$$E(H)^1=s^{-1}_E(H)\cup \overline{F}(H).$$
$s_{E(H)}(e)=s_E(e)$ and $r_{E(H)}(e)=r_E(e)$ for each
$e\in s^{-1}_E(H)$. $s_{E(H)}(\overline{\alpha})=\alpha$ and $r_{E(H)}(\overline{\alpha})=r_E(\alpha)$ for each $\overline{\alpha}\in \overline{F}(H)$.
}
\end{definition}

\begin{example}\label{exa} Consider the graph $E$ given by
$$
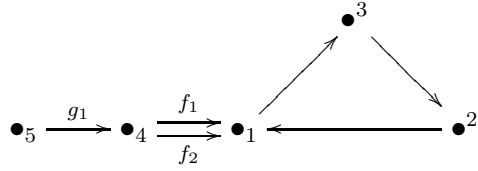
\xymatrix{
&&&{\bullet}^3\ar[dr]&\\
{\bullet}_5 \ar[r]^{g_1}&{\bullet}_{4}\ar@<.5ex>[r]^{f_1}\ar@<-.5ex>[r]_{f_2}&{\bullet}_{1}\ar[ur]&&{\bullet}^{2}\ar[ll]\\
}$$

Let $H=\{1, 2, 3\}$. Thus $F(H)=\{f_1, f_2, g_1f_1, g_1f_2\}$ and $E(H)$ is the graph\\

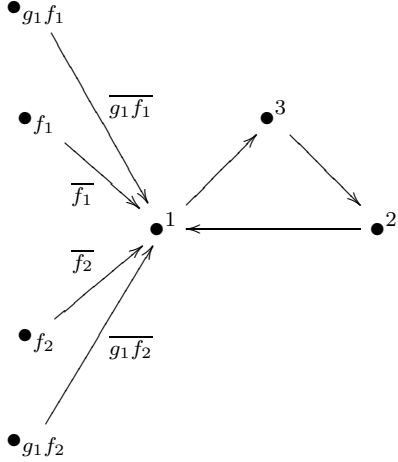
$$\xymatrix{
{\bullet}_{g_1f_1}\ar[ddr]^{\overline{g_1f_1}}&&&\\
{\bullet}_{f_1}\ar[dr]_{\overline{f_{1}}}&&{\bullet}^{3}\ar[dr]&\\
&{\bullet}^{1}\ar[ur]&&{\bullet}^{2}\ar[ll]\\
{\bullet}_{f_2}\ar[ur]^{\overline{f_{2}}}&&&\\
{\bullet}_{g_1f_2}\ar[uur]_{\overline{g_1f_{2}}}&&&
}$$
\end{example}

\begin{theorem}\label{thm1} Let $R$ be a commutative unital ring, $E$ be a graph and $H$ be a hereditary subset of $E^0$. Suppose that $(E^0\backslash H, r^{-1}_E(E^0\backslash H), r_E, s_E)$ is a finite acyclic graph, $v\geq H$ for all $v\in E^0\backslash H$ and the set $s_E^{-1}(E^0\backslash H)\cap r^{-1}_E(H)$ is finite. Then $L_R(E)\cong L_R(E(H))$.
\end{theorem}

\begin{proof} Let $\{e, v|e\in E^1, v\in E^0\}$ be a universal Cuntz-Krieger $E$-family. For $v\in E(H)^0$ define
\begin{align*}
Q_v  & = \begin{cases}
v & \text{if }\ v\in H,\\
\alpha\alpha^* & \text{if}\ v=\alpha\in F(H) \end{cases}
\end{align*}
and for $e\in E(H)^1$ define
\begin{align*}
T_e & = \begin{cases}
e & \text{if }\ e\in s^{-1}_E(H),\\
\alpha & \text{if }\ e=\overline{\alpha}\in \overline{F}(H).  \end{cases}
\end{align*}
The same argument as in the proof of the Theorem 3.8 of \cite{AR} (see also Lemma 3.7 of \cite{Na}) shows that $\{T_e, Q_v|e\in E(H)^1, v\in E(H)^0\}$ is a Cuntz-Krieger $E(H)$-family in $L_R(E)$. Let $\{t_e, q_v|e\in E(H)^1, v\in E(H)^0\}$ be a universal Cuntz-Krieger $E(H)$-family. By the universal homomorphism property of $L_R(E(H))$ there exists a $*$-homomorphism $\Psi:L_R(E(H))\rightarrow L_R(E)$ with $\Psi(q_v)=Q_v$ for each $v\in E(H)^0$ and $\Psi(t_e)=T_e$ for each $e\in E(H)^1$. Since $s_E^{-1}(E^0\backslash H)\cap r^{-1}_E(H)$ is finite, the same argument as in the proof of the Theorem 3.8 of \cite{AR} shows that $\Psi$ is epimorphism. Now let $\alpha$ be a distinguished path in $E(H)$ and $\omega_\alpha=\alpha\lambda_\alpha\alpha^*$, where $\lambda_\alpha$ is a cycle without exits that starts and ends at $r_{E(H)}(\alpha)$. The cycles in $E(H)$ come from cycles in $E$ all lying in the subgraph $(H, s^{-1}_E(H), s_E, r_E)$. Hence $\lambda_\alpha$ is a cycle without exits in $E$ that starts and ends at $r_{E(H)}(\alpha)$. If $\alpha$ is a distinguished vertex, then $\omega_\alpha=\lambda_\alpha$ and $\Psi(\omega_\alpha)=\omega_\alpha$. If $l(\alpha)\neq 0$, then $\alpha=\mu\overline{\beta}$ for some $\mu=e_1\cdots e_t\in \text{Path}(E)\backslash F(H)$ and $\beta\in F(H)$. Therefore $\Psi(\omega_\alpha)=\mu\beta\lambda_\alpha\beta^*\mu^*$ and so $<\Psi(\omega_\alpha)>\cong R[x, x^{-1}]$. Now let $v\in E(H)^0$ and $r\in R\backslash \{0\}$. If $v\in H$, then $\Psi(rq_v)=rv$ and by \cite[Proposition 3.4]{T}, $rv\neq 0$. Now assume that $v=\alpha$ for some $\alpha\in F(H)$. Hence $\Psi(rq_v)=r\alpha\alpha^*$ and by \cite[Proposition 4.9]{T}, $r\alpha\alpha^*\neq 0$. Thus by the generalized uniqueness theorem, $\Psi$ is injective. Therefore $\Psi$ is an isomorphism and the result follows.
\end{proof}

\begin{definition}{\cite[Definitions 3.2, 3.3 and 3.9]{AR}}\label{def2} {\rm
Let $E$ be a graph and $n$ be a positive integer.
\begin{itemize}
\item[(i)] For any vertex $v_0\in E^0$ define a graph $E(v_0, n)$ as follows:
$$E(v_0, n)^0=E^0\cup \{v_1, v_2, \cdots, v_n\},$$
$$E(v_0, n)^1=E^1\cup \{e_1, e_2, \cdots, e_n\}$$
$s_{E(v_0, n)}(e)=s_E(e)$ and $r_{E(v_0, n)}(e)=r_E(e)$ for each
$e\in E^1$. $r_{E(v_0, n)}(e_i)=v_{i-1}$ and $s_{E(v_0, n)}(e_i)=v_i$ for each $i$.
\item[(ii)] For each edge $e_0\in E^1$ define a graph $E(e_0, n)$ as follows:
$$E(e_0, n)^0=E^0\cup \{v_1, v_2, \cdots, v_n\},$$
$$E(e_0, n)^1=\{e_1, e_2, \cdots, e_{n+1}\}\cup E^1\backslash \{e_0\} $$
$s_{E(e_0, n)}(e)=s_E(e)$ and $r_{E(e_0, n)}(e)=r_E(e)$ for each
$e\in E^1\backslash \{e_0\}$. $r_{E(e_0, n)}(e_i)=v_{i-1}$ for each $2\leq i\leq n+1$, $s_{E(e_0, n)}(e_i)=v_i$ for each $1\leq i\leq n$, $r_{E(e_0, n)}(e_1)=r_E(e_0)$ and $s_{E(e_0, n)}(e_{n+1})=s_E(e_0)$.
\item[(iii)] For any vertex $v_0\in E^0$ define a graph $E'(v_0, n)$ as follows:
$$E'(v_0, n)^0=E^0\cup \{v_1, v_2, \cdots, v_n\},$$
$$E'(v_0, n)^1=E^1\cup \{e_1, e_2, \cdots, e_n\}$$
$s_{E'(v_0, n)}(e)=s_E(e)$ and $r_{E'(v_0, n)}(e)=r_E(e)$ for each
$e\in E^1$. $r_{E'(v_0, n)}(e_i)=v_{0}$ and $s_{E'(v_0, n)}(e_i)=v_i$ for each $i$.
\end{itemize}
}
\end{definition}

\begin{example}\label{exa1} Consider the graph $E$ given by
$$\xymatrix{
&&&{\bullet}_u\ar[dr]^{\beta}&\\
&&{\bullet}_{v}\ar[ur]^{\alpha}&&{\bullet}_{w}\ar[ll]^{\gamma}\\
}$$

Thus $E(v, 3)$ is the graph\\

$$\xymatrix{
&&&&{\bullet}_u\ar[dr]^{\beta}&\\
{\bullet}_{v_3}\ar[r]^{e_3}&{\bullet}_{v_2}\ar[r]^{e_2}&{\bullet}_{v_1}\ar[r]^{e_1}&{\bullet}_{v}\ar[ur]^{\alpha}&&{\bullet}_{w}\ar[ll]^{\gamma}\\
}$$

$E(\alpha, 3)$ is the graph\\

$$\xymatrix{
{\bullet}_{v_2}\ar[r]^{e_2}&{\bullet}_{v_1}\ar[r]^{e_1}&{\bullet}_u\ar[dr]^{\beta}&\\
{\bullet}_{v_3}\ar[u]^{e_3}&{\bullet}_{v}\ar[l]^{e_4}&&{\bullet}_{w}\ar[ll]^{\gamma}\\
}$$

and $E'(v, 3)$ is the graph\\

$$\xymatrix{
&{\bullet}_{v_1}\ar[dr]^{e_1}&&{\bullet}_u\ar[dr]^{\beta}&\\
&{\bullet}_{v_2}\ar[r]^{e_2}&{\bullet}_{v}\ar[ur]^{\alpha}&&{\bullet}_{w}\ar[ll]^{\gamma}\\
&{\bullet}_{v_3}\ar[ur]_{e_3}&&&\\
}$$
\end{example}

\begin{corollary}\label{cor0} Let $R$ be a commutative unital ring, $E$ be a graph, $v_0\in E^0$ be a vertex and $n$ be a positive integer. Then $L_R(E(v_0, n))\cong L_R(E'(v_0, n))$.
\end{corollary}

\begin{proof} The similar argument as in the proof of the Corollary 3.10 of \cite{AR} shows that the result follows from Theorem \ref{thm1}.
\end{proof}

\begin{proposition}\label{prop0} Let $R$ be a commutative unital ring, $E$ be a graph, $e_0\in E^1$ be an edge and $n$ be a positive integer. Then $L_R(E(r_E(e_0), n))\cong L_R(E(e_0, n))$.
\end{proposition}

\begin{proof} Let $r_E(e_0)=v_0$ and $\{e, v|e\in E(e_0, n)^1, v\in E(e_0, n)^0\}$ be a universal Cuntz-Krieger $E(e_0, n)$-family. For $v\in E(v_0, n)^0$ define
$Q_v=v$ and for $e\in E(v_0, n)^1$ define
\begin{align*}
T_e & = \begin{cases}
e & \text{if }\ e\neq e_0,\\
e_{n+1}e_{n}\cdots e_1 & \text{if }\ e=e_0.  \end{cases}
\end{align*}
The same argument as in the proof of the Proposition 3.5 of \cite{AR} shows that $\{T_e, Q_v|e\in E(v_0, n)^1, v\in E(v_0, n)^0\}$ is a Cuntz-Krieger $E(v_0, n)$-family in $L_R(E(e_0, n))$. Let $\{t_e, q_v|e\in E(v_0, n)^1, v\in E(v_0, n)^0\}$ be a universal Cuntz-Krieger $E(v_0, n)$-family. By the universal homomorphism property of $L_R(E(v_0, n))$ there exists a $*$-homomorphism \\$\Psi:L_R(E(v_0, n))\rightarrow L_R(E(e_0, n))$ that $\Psi(q_v)=Q_v$ for each $v\in E(v_0, n)^0$ and $\Psi(t_e)=T_e$ for each $e\in E(v_0, n)^1$. The same argument as in the proof of the Proposition 3.5 of \cite{AR} shows that $\Psi$ is an epimorphism.

Now let $\alpha$ be a distinguished path in $E(v_0, n)$ and $\omega_\alpha=\alpha\lambda_\alpha\alpha^*$, where $\lambda_\alpha$ is a cycle without exits that starts and ends at $r_{E(v_0, n)}(\alpha)$. Suppose $\lambda_\alpha=f_1f_2\cdots f_m$. If $s_{E(v_0, n)}(f_i)\neq s_{E(v_0, n)}(e_0)$ for each $i$, then $\lambda_\alpha$ is a cycle without exits in $E(e_0, n)$. Thus $\Psi(t_{\lambda_\alpha})=\lambda_\alpha$ and so $<\Psi(\omega_\alpha)>\cong R[x, x^{-1}]$. Now assume that $s_{E(v_0, n)}(f_i)=s_{E(v_0, n)}(e_0)$, for some $i$. Since $\lambda_\alpha$ is a cycle without exits, $\lambda_\alpha=e_0f_{i+1}f_{i+2}\cdots f_{i-1}$. $\Psi(t_{\lambda_\alpha})=e_{n+1}e_n\cdots e_1f_{i+1}f_{i+2}\cdots f_{i-1}$ and $e_{n+1}e_n\cdots e_1f_{i+1}f_{i+2}\cdots f_{i-1}$ is a cycle without exits in $E(e_0, n)$. Thus $<\Psi(\omega_\alpha)>\cong R[x, x^{-1}]$. Also for each $v\in E(v_0, n)^0$ and $r\in R\backslash \{0\}$, $\Psi(rq_v)=rv\neq 0$. It follows from the generalized uniqueness theorem that $\Psi$ is injective. Therefore $\Psi$ is an isomorphism and the result follows.
\end{proof}

When $E$ is a row-finite graph with no sinks and $k$ is a field, Proposition 3.1 of \cite{ALPS} shows that, there exists a row-finite graph $G$ with no sinks and no sources such that the Leavitt path algebras $L_k(E)$ and $L_k(G)$ are Morita equivalent. Also when $E$ is a finite graph with no sinks and at least two vertices, Proposition 13 of \cite{H1} shows that, there exists a finite graph $G$ with no sinks and no sources such that the Leavitt path algebras $L_k(E)$ and $L_k(G)$ are graded Morita equivalent (See also \cite[Corollary 3.18]{N}). The following corollary improves
these known (graded) Morita equivalences to isomorphisms.

\begin{corollary}\label{cor1} Let $R$ be a commutative unital ring and $E$ be a finite graph with no sinks. Then there exists a finite graph $G$ with no sinks and no sources such that the Leavitt path algebras $L_R(E)$ and $L_R(G)$ are isomorphic.
\end{corollary}

\begin{proof} Let $E_0=E$ and remove the sources of $E_0$ we get a subgraph $E_1$ of $E_0$. Remove the sources of $E_1$, we get a subgraph $E_2$ of $E_1$ (see \cite[Definition 1.2]{ALPS}). Since $E$ is a finite graph with no sinks, after finitely many times, we get a subgraph $F=E_n$ of $E$ that has no sinks and no sources. By induction we see that $F^0$ is a hereditary subset of $E^0$. We show that $(E^0\backslash F^0, r_E^{-1}(E^0\backslash F^0), r_E, s_E)$ is a finite acyclic graph that for any $v\in E^0\backslash F^0$ there exists a path from $v$ to $F^0$ in $E$. Since $E$ is a finite graph, $(E^0\backslash F^0, r_E^{-1}(E^0\backslash F^0), r_E, s_E)$ is finite.
Let $e_1e_2\cdots e_r$ be a cycle in $(E^0\backslash F^0, r_E^{-1}(E^0\backslash F^0), r_E, s_E)$. Then $e_1e_2\cdots e_r$ is a cycle in $E$ and so it is a cycle in $E_1$. Inductively, $e_1e_2\cdots e_r$ is a cycle in $E_i$ for each $0\leq i\leq n$. Since $s_E(e_1)\in E^0\backslash F^0$, $s_E(e_1)$ is a source in $E_j$ for some $0\leq j\leq n$ which is a contradiction. Therefore  $(E^0\backslash F^0, r_E^{-1}(E^0\backslash F^0), r_E, s_E)$ is acyclic. Let $v\in E^0\backslash F^0$. Then there exists $0\leq j\leq n$ such that $v$ is a source in $E_j$. Since $E$ has no sinks, there exists an edge $e_1\in E^1$ such that $s_E(e_1)=v$ and $r_E(e_1)\in E_{j+1}^{0}$. If $j+1=n$, then $e_1$ is a path from $v$ to $F^0$. Assume that $j+1<n$. If $r_E(e_1)$ is not a source in $E_{i}$ for each $i\geq j+1$, then $r_E(e_1)\in F^0$ and $e_1$ is a path from $v$ to $F^0$. If $r_E(e_1)$ is a source in $E_{i}$ for some $i\geq j+1$, then there exists an edge $e_{2}\in E^1$ such that $s_E(e_{2})=r_E(e_1)$ and $r_E(e_{2})\in E_{i+1}^{0}$. Continuing in this way, since $E$ is a finite graph, we get a path from $v$ to $F^0$. Thus by Theorem \ref{thm1}, $L_R(E)\cong L_R(E(F^0))$. By definition, $E(F^0)^0=F^0\cup F(F^0)$. Since $(E^0\backslash F^0, r_E^{-1}(E^0\backslash F^0), r_E, s_E)$ is a finite acyclic graph, $F(F^0)=\{\alpha|\alpha=e_1e_2\cdots e_n\in \text{Path}(E), s_E(e_n)\not\in F^0, r_E(e_n)\in F^0 \}$ is a finite set. Assume that $F(F^0)=\{\alpha_1, \alpha_2,\cdots, \alpha_p\}$ for some positive integer $p$. $E(F^0)^1=s^{-1}_E(F^0)\cup \{\overline{\alpha_1}, \overline{\alpha_2},\cdots, \overline{\alpha_p}\}$. Let $r_{E(F^0)}(\overline{\alpha_1})=r_E(\alpha_1)=v_1\in F^0$ and assume that $r^{-1}_{E(F^0)}(v_1)=r^{-1}_F(v_1)\cup \{\overline{\alpha_{1_1}}, \overline{\alpha_{1_2}},\cdots, \overline{\alpha_{1_i}}\}$ for some $1\leq i\leq p$, where $\overline{\alpha_{1_1}}=\overline{\alpha_1}$. Remove the vertices $\alpha_{1_1}, \alpha_{1_2}, \cdots, \alpha_{1_i}$ of $E(F^0)$, we get a graph $G_1$ such that $G_1'(v_1, i)=E(F^0)$. By Corollary \ref{cor0}, $L_R(E(F^0))\cong L_R(G_1(v_1, i))$. $F$ has no sources, then there exists an edge $e_1\in F^1$ such that $r_F(e_1)=r_E(e_1)=v_1$ and so by Proposition \ref{prop0}, $L_R(G_1(v_1, i))\cong L_R(G_1(e_1, i))$. Therefore $L_R(E(F^0))\cong L_R(G_1(e_1, i))$. The above procedure shows that $G_1(e_1, i)$ is a finite graph with no sinks and with $p-i$ sources. Continuing in this way, after finitely many steps we get a finite graph $G$ with no sinks and no sources that $L_R(E)\cong L_R(G)$.
\end{proof}

\begin{corollary}\label{cor2} Let $E$ be a graph. Then the following are equivalent:
\begin{itemize}
\item[(1)] $E$ is a finite graph with no sinks;
\item[(2)] For every commutative unital ring $R$, $L_R(E)$ is isomorphic to an algebraic Cuntz-Krieger algebra;
\item[(3)] There exists a field $k$ such that $L_k(E)$ is isomorphic to an algebraic Cuntz-Krieger algebra;
\item[(4)] $E$ is a finite graph and for every commutative unital ring $R$, $L_R(E)$ is strongly $\mathbb{Z}$-graded;
\item[(5)] $C^*(E)$ is unital and $rank(K_0(C^*(E)))=rank(K_1(C^*(E)))$.
\end{itemize}
\end{corollary}

\begin{proof} $(1) \Rightarrow (2)$ follows from Corollary \ref{cor1}.

$(2) \Rightarrow (3)$ is obvious.

For $(3) \Rightarrow (1)$, suppose that there exist a field $k$ and a finite graph $E'$ with no sinks and no sources such that $L_k(E)\cong\mathcal{CK}_k(E')$. Hence $L_k(E)$ is unital and so $E$ has a finite number of vertices. Then, by \cite[Corollary 6.17]{RT}, $E$ has no singular vertices. Therefore $E$ is a finite graph with no sinks and the result follows.

$(1) \Leftrightarrow (4)$ follows from \cite[Theorem 3.15]{H3}.

$(1) \Leftrightarrow (5)$ follows from \cite[Theorem 3.12]{AR}.
\end{proof}

Let $(G, +)$ be an abelian group. A finite set of elements $\{g_1, g_2, \cdots, g_l\}\subseteq G$ is called linearly independent if whenever $\sum_{i=1}^{l}n_ig_i=0$ for $n_1, \cdots, n_l\in \mathbb{Z}$, then $n_i=0$ for each $1\leq i\leq l$. Any two maximal linearly independent sets in $G$ have the same cardinality. If there exits a maximal linearly independent set in $G$, the cardinality of this set is called the rank $rank(G)$ of $G$ and if there is no maximal linearly independent set in $G$, the rank $rank(G)$ of $G$ is defined to be infinite.

Now we are ready to proof the main result of this section.

\begin{theorem}\label{thm2} Let $k$ be a field such that $rank(k^{\times})< \infty$ and $E$ be a graph. Then the following are equivalent:
\begin{itemize}
\item[(1)] $E$ is a finite graph with no sinks;
\item[(2)] $L_k(E)$ is isomorphic to an algebraic Cuntz-Krieger algebra;
\item[(3)] $E$ is a finite graph and $L_k(E)$ is strongly $\mathbb{Z}$-graded;
\item[(4)] $C^*(E)$ is unital and $rank(K_1(C^*(E)))=rank(K_0(C^*(E)))$;
\item[(5)] $L_k(E)$ is unital and $rank(K_1(L_k(E)))=(rank(k^{\times})+1)rank(K_0(L_k(E)))$.
\end{itemize}
\end{theorem}

\begin{proof} $(1)\Leftrightarrow(2)\Leftrightarrow(3)\Leftrightarrow(4)$ follows from Corollary \ref{cor2}.

For $(2) \Rightarrow (5)$, suppose that there exists a finite graph $E'$ with no sinks and no sources such that $L_k(E)\cong\mathcal{CK}_k(E')$. Hence $L_k(E)$ is unital and since $rank(k^{\times})< \infty$, by \cite[Theorem 8.1]{GRTW} we have $|E'^0_{\text{sing}}|=(rank(k^{\times})+1)rank(K_0(L_k(E')))-rank(K_1(L_k(E')))$. Thus $(rank(k^{\times})+1)rank(K_0(L_k(E')))=rank(K_1(L_k(E')))$ and so \\$(rank(k^{\times})+1)rank(K_0(L_k(E)))=rank(K_1(L_k(E)))$.

For $(5) \Rightarrow (1)$, suppose that $L_k(E)$ is unital and $rank(K_1(L_k(E)))=\\(rank(k^{\times})+1)rank(K_0(L_k(E)))$. Thus $E^0$ is a finite set and by \cite[Theorem 8.1]{GRTW} we have $|E^0_{\text{sing}}|=(rank(k^{\times})+1)rank(K_0(L_k(E)))-rank(K_1(L_k(E)))$. Since $rank(k^{\times})< \infty$, $rank(K_1(L_k(E)))< \infty$ by \cite[Theorem 8.1]{GRTW}. Thus $|E^0_{\text{sing}}|=0$ and so $E$ has no singular vertices. Therefore $E$ is a finite graph with no sinks and the result follows.
\end{proof}

The following example shows that the assumption $rank(k^{\times})< \infty$ in the Theorem \ref{thm2} is necessary.

\begin{example}\label{exa}
Let $E$ be the graph $\begin{matrix}\xymatrix{{\bullet}_1\ar[r]^{\alpha}&{\bullet}_2}
\end{matrix}$ and $\mathbb{Q}$ be the field of rational numbers. $rank(\mathbb{Q}^{\times})= \infty$ and $K_1(\mathbb{Q})\cong \mathbb{Q}^{\times}$. $L_{\mathbb{Q}}(E)\cong \mathbb{M}_2(\mathbb{Q})$ and so $K_1(L_{\mathbb{Q}}(E))\cong K_1(\mathbb{Q})\cong \mathbb{Q}^{\times}$. Thus $rank(K_1(L_{\mathbb{Q}}(E)))=(rank(\mathbb{Q}^{\times})+1)rank(K_0(L_{\mathbb{Q}}(E)))=\infty$, $L_{\mathbb{Q}}(E)$ is unital and $E$ is a finite graph but $E$ has a sink.
\end{example}


\section{Corners of Leavitt path algebras}

In this section we show that there exists a graph $E(T)$ for the corner $P_XL_R(E)P_X$ of a Leavitt path algebra $L_R(E)$ associated to a finite vertex set $X$, such that $P_XL_R(E)P_X\cong L_R(E(T))$.

Let $E$ be a graph. An acyclic subgraph $T$ of $E$ is called a directed forest in $E$ if for each $v\in T^0$, $|T^1\cap r_E^{-1}(v)|\leqslant 1$. We denote
by $T^r$ the subset of $T^0$ consisting of those vertices $v$ with $|T^1\cap r_E^{-1}(v)|=0$ and by $T^l$ the subset of $T^0$ consisting of those vertices $v$ with $|T^1\cap s_E^{-1}(v)|=0$.

\begin{definition}\label{def1}(\cite[Definition 3.1]{C})
Let $E$ be a graph, $X\subsetneqq E^0$ be a finite set and $T$ be a row-finite, path-finite directed forest in $E$ with $T^r=X$ and $T^0=H_E(X)$. Define the $T$-corner, $E(T)$ of $E$ as follows:

$$E(T)^0 =T^0\backslash \{v| v\in T^0, \varnothing\neq s_E^{-1}(v)\subseteq T^1\},$$
$$E(T)^1 = \{e_u| e\in s_E^{-1}(T^0)\backslash T^1, u\in E(T)^0, r_E(e)\geq_T u\} ,$$
$$s_{E(T)}(e_u)=s_E(e),  r_{E(T)}(e_u)=u.$$
\end{definition}

Let $E$ be a graph, $X\subsetneqq E^0$ be a finite set. According to \cite[Lemma 3.6]{C} there is a forest $T$ in $E$ which satisfies the conditions of \ref{def1} if and only if $H_E(X)$ is finite.

\begin{example}\label{example 3}
Let
$\begin{matrix}E:\xymatrix{{\bullet_{v_1}}\ar@/^1pc/[r]^{\gamma}&{\bullet_{v_2}}\ar@/^1pc/[r]^{\alpha}\ar@/^1pc/[l]^{\delta}&{\bullet_{v_3}}\ar@/^1pc/[l]^{\beta}}
\end{matrix},$ $X=\{v_2\}$  and

 $$\begin{matrix}T: 
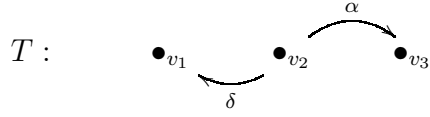
\xymatrix{
&{\bullet_{v_1}}&{\bullet_{v_2}}\ar@/^1pc/[r]^{\alpha}\ar@/^1pc/[l]^{\delta}&{\bullet_{v_3}}}
\end{matrix}$$

$T$ is a row-finite, path-finite directed forest in $E$ with $T^r=\{v_2\}$ and $T^0=H_E(\{v_2\})$. Thus $E(T)$ is the following graph:

$$\begin{matrix}E(T): 
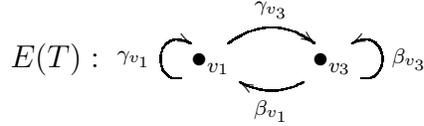
\xymatrix{{\bullet_{v_1}}\ar@(dl,ul)^{\gamma_{v_1}}\ar@/^1pc/[r]^{\gamma_{v_3}}&{\bullet_{v_3}}\ar@(dr,ur)_{\beta_{v_3}}\ar@/^1pc/[l]^{\beta_{v_1}}}
\end{matrix}$$
\end{example}

Crisp in \cite[Theorem 3.5]{C} proved that $C^*(E(T))\cong P_XC^*(E)P_X$, where $\sum_{v\in X}P_v$ converges strictly to a projection $P_X$ in the multiplier algebra $M(C^*(E))$. In the following theorem we prove the similar result for Leavitt path algebras.

\begin{theorem}\label{thm3} Let $E$ be a graph, $X\subsetneqq E^0$ be a finite set, $T$ be a row-finite, path-finite directed forest in $E$ with $T^r=X$ and $T^0=H_E(X)$, $P_X=\sum_{x\in X}x$ and $R$ be a commutative unital ring. Then $L_R(E(T))\cong P_XL_R(E)P_X$. If in addition $L_R(E)$ is an algebraic Cuntz-Krieger algebra, then $L_R(E(T))$ is isomorphic to an algebraic Cuntz-Krieger algebra.
\end{theorem}

\begin{proof} Let $\{e, v|e\in E^1, v\in E^0\}$ be the universal Cuntz-Krieger $E$-family. By \cite[Lemma 2.1 (i)]{C}, for any $v\in T^0$ there exists a unique path $\tau(v)$ in $\text{Path}(T)$ such that $s_T(\tau(v))\in T^r$ and $r_T(\tau(v))=v$. For each $v\in T^0$, let $Q_v=\tau(v)\tau(v)^*-\sum_{e\in T^1\cap s^{-1}(v)}\tau(v)ee^*\tau(v)^*$. For each $e_u\in E(T)^1$, let $T_{e_u}=\tau(s(e))e\tau(r(e))^*Q_u$. The similar argument as in the proof of the \cite[Proposition 3.8]{C} shows that $\{T_{e_u}, Q_v|e_u\in E(T)^1, v\in E(T)^0\}$ is a Cuntz-Krieger $E(T)$-family in $E$. By the universal homomorphism property of $L_R(E(T))$ there exists a $*$-homomorphism $\Psi:L_R(E(T))\rightarrow L_R(E)$ with $\Psi(v)=Q_v$ for each $v\in E(T)^0$ and $\Psi(e_u)=T_{e_u}$ for each $e_u\in E(T)^1$. Let $w:E^1\rightarrow \mathbb{Z}$ be a weight map given by
\begin{align*}
w(e) & = \begin{cases}
l(\tau(r(e)))-l(\tau(s(e)))+1 & \text{if }\ e\not\in T^1; s(e), r(e)\in T^0,\\
1 & \text{}\ otherwise.  \end{cases}
\end{align*}

Let $w(e^*)=-w(e)$ and $w(v)=0$ for each $e\in E^1$ and $v\in E^0$. Thus $L_R(E)$ is a $\mathbb{Z}$-graded algebra. Also $L_R(E(T))$ is a $\mathbb{Z}$-graded algebra with the usual $\mathbb{Z}$-grading. We show that $\Psi$ is a $\mathbb{Z}$-graded ring homomorphism. For each $v\in E(T)^0$, $\Psi(v)=Q_v=\tau(v)\tau(v)^*-\sum_{e\in T^1\cap s^{-1}(v)}\tau(v)ee^*\tau(v)^*$ and so the degree of $\Psi(v)$ is zero. For each $e_u\in E(T)^1$, $\Psi(e_u)=T_{e_u}=\tau(s(e))e\tau(r(e))^*Q_u$. $\tau(s(e)), \tau(r(e))\in \text{Path}(T)$ and so $\tau(s(e))e\tau(r(e))^*Q_u$ is homogeneous of degree $l(\tau(s(e)))+l(\tau(r(e)))-l(\tau(s(e)))+1-l(\tau(r(e)))=1$. Thus $\Psi$ is a $\mathbb{Z}$-graded ring homomorphism.

Now we show that for each $v\in E(T)^0$ and each $r\in R\backslash\{0\}$, $\Psi(rv)\neq 0$. Let $v\in E(T)^0$ and $r\in R\backslash\{0\}$. $\Psi(rv)=rQ_v=r(\tau(v)\tau(v)^*-\sum_{e\in T^1\cap s^{-1}(v)}\tau(v)ee^*\tau(v)^*)$. By Definition \ref{def1}, $v$ is either sink in $E$ or $v$ emits an edge $f\in E^1\backslash T^1$. If $v$ is a sink in $E$, then $Q_v=\tau(v)\tau(v)^*$ and by \cite[Proposition 4.9]{T} for each $r\in R\backslash\{0\}$, $rQ_v\neq 0$. Now assume that $v$ emits an edge $f\in E^1\backslash T^1$. If $rQ_v=0$ for some $r\in R\backslash\{0\}$, then $r\tau(v)\tau(v)^*-r\sum_{e\in T^1\cap s^{-1}(v)}\tau(v)ee^*\tau(v)^*=0$. Thus $(r\tau(v)\tau(v)^*-r\sum_{e\in T^1\cap s^{-1}(v)}\tau(v)ee^*\tau(v)^*)\tau(v)ff^*\tau(v)^*=0$. Since $e\in T^1$ and $f\not\in T^1$, $e^*f=0$. Hence $r\tau(v)ff^*\tau(v)^*=0$ and so $r\tau(v)f=0$. This leads to a contradiction with the Proposition 4.9 of \cite{T}. Thus $\Psi(rv)\neq 0$ for each $v\in E(T)^0$ and each $r\in R\backslash\{0\}$ and by the graded uniqueness theorem, $\Psi$ is injective.  The similar argument as in the proof of the \cite[Proposition 3.11]{C} shows that $\Psi(L_R(E(T)))=P_XL_R(E)P_X$ and so $L_R(E(T))\cong P_XL_R(E)P_X$.

Now suppose that $L_R(E)$ is an algebraic Cuntz-Krieger algebra. Thus $E$ is a finite graph with no sinks and no sources. Since $E$ is a finite graph, any directed forest in $E$ is finite and so $E(T)$ is a finite graph. Assume, on the contrary, that $v\in E(T)^0$ and $s_{E(T)}^{-1}(v)=\emptyset$. Since $E$ has no sinks, $s_E^{-1}(v)\neq \emptyset$. Thus by the definition of $E(T)$, $s_E^{-1}(v)\not\subseteq T^1$. Therefore there exists $e\in E^1\backslash T^1$ such that $s_E(e)=v$. If $r_E(e)\in E(T)^0$, then $e_{r_E(e)}\in E(T)^1$ and $s_{E(T)}(e_{r_E(e)})=v$, which is a contradiction. Thus $r_E(e)\not\in E(T)^0$. Since $T^0=H_E(X)$ and $v\in T^0$, $r_E(e)\in T^0$. Also $s_E^{-1}(r_E(e))\neq \emptyset$, then there exists $e_1\in T^1$ with $s_E(e_1)=r_E(e)$. Let $r_E(e_1)=v_1$. If $v_1\in E(T)^0$, then $s_{E(T)}(e_{v_1})=v$, which is a contradiction. Thus $v_1\in T^0\backslash E(T)^0$. Similar argument shows that there exists $e_2\in T^1$ with $s_E(e_2)=v_1$. Since $T$ is an acyclic graph, by continuing in this way we get an infinite path $e_1e_2e_3\cdots$ in $T$ which is a contradiction. Thus $E(T)$ is a finite graph with no sinks and by Corollary \ref{cor2}, $L_R(E(T))$ is isomorphic to an algebraic Cuntz-Krieger algebra.
\end{proof}

Let $E$ be a finite graph with no sinks and no sources. In the proof of the above Theorem we show that $E(T)$ is a finite graph with no sinks. The following example shows that there exists a finite graph $E$ with no sinks and no sources such that $E(T)$ has a source.

\begin{example}
Let $E$ be the graph
$$
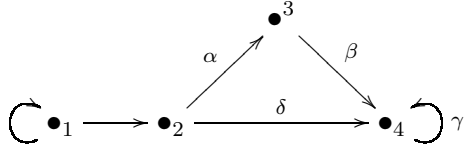
\xymatrix{
&&{\bullet^{3}}\ar[dr]^{\beta}&\\
{\bullet_{1}}\ar@(dl,ul)\ar[r]&{\bullet_{2}}\ar[ur]^{\alpha}\ar[rr]^{\delta}&&{\bullet_{4}}\ar@(dr,ur)_{\gamma}}\\
$$, $X=\{2\}$ and $T$ be the directed forest
$$\xymatrix{
&&{\bullet^{3}}\ar[dr]^{\beta}&\\
&{\bullet_{2}}\ar[ur]^{\alpha}&&{\bullet_{4}}}\\
$$
Thus $E(T)$ is the following graph:
$$\xymatrix{{\bullet_{2}}\ar[rr]^{\delta_{4}}&&{\bullet_{4}}\ar@(dr,ur)_{\gamma_{4}}}
$$
\end{example}

\begin{definition}(\cite[Definitions 9.1 and 9.4]{AT})
Let $E$ be a graph.
\begin{itemize}
\item[(1)]
Define $M_nE$ to be the graph formed from $E$ by taking
each $v\in E^0$ and attaching a head of length $n-1$ of the form
$$\begin{matrix}
\xymatrix{{\bullet}_{v_{n-1}} \ar[r]^{e_{n-1}^v} & {\bullet}_{v_{n-2}} \ar[r]^{e_{n-2}^v}  & {\bullet}_{v_{n-3}} \ar @{.>}[r]&{\bullet}_{v_2} \ar [r]^{e_2^v} &{\bullet}_{v_1} \ar [r]^{e_1^v}  & {\bullet}_{v}}
\end{matrix}$$
to $E$.
\item[(2)]
Define $SE$ to be the graph formed from $E$ by taking
each $v\in E^0$ and attaching an infinite head of the form
$$\begin{matrix}
\xymatrix{ {\cdots} \ar[r] & {\bullet}_{v_{3}} \ar[r]^{e_{3}^v}&{\bullet}_{v_2} \ar [r]^{e_2^v} &{\bullet}_{v_1} \ar [r]^{e_1^v}  & {\bullet}_{v}}
\end{matrix}$$
to $E$. $SE$ is called the stabilization of $E$.
\end{itemize}
\end{definition}

Let $R$ be a ring. The ring of finitely supported, countably infinite square matrices with coefficients in $R$ is denoted by $\mathbb{M}_\infty(R)$ \cite[Definition 9.6]{AT}. Note that if $R$ is an algebra (resp. a $*$-algebra), then $\mathbb{M}_\infty(R)$ inherits an algebra (resp. a $*$-algebra) structure. Algebras $A$ and $B$ are called stably isomorphic if $\mathbb{M}_\infty(A)\cong \mathbb{M}_\infty(B)$. Abrams and Tomforde in \cite[Propositions 9.3 and 9.8]{AT} proved that, for any graph $E$ and field $k$, $L_k(M_nE)\cong \mathbb{M}_n(L_k(E))$ and $L_k(SE)\cong \mathbb{M}_\infty(L_k(E))$. A similar argument as in the proof of \cite[Propositions 9.3 and 9.8]{AT} with a commutative unital ring $R$ in place of field $k$, shows that $L_R(M_nE)\cong \mathbb{M}_n(L_R(E))$ and $L_R(SE)\cong \mathbb{M}_\infty(L_R(E))$.

\begin{corollary}\label{cor3} Let $E$ be a graph and $R$ be a commutative unital ring.
\begin{itemize}
\item[(1)] Let $X$ be a finite subset of $(SE)^0$ such that $H_{SE}(X)$ is a finite set and $e_X=\sum_{x\in X}x$. Then there exists a row-finite, path-finite directed forest $T$ in $SE$ such that $L_R(SE(T))\cong e_XL_R(SE)e_X$. If in addition $L_R(SE)$ is an algebraic Cuntz-Krieger algebra, then $e_XL_R(SE)e_X$ is isomorphic to an algebraic Cuntz-Krieger algebra.
\item[(2)] Let $n$ be a positive integer, $X$ be a finite subset of $(M_nE)^0$ such that $H_{M_nE}(X)$ is a finite set and $e_X=\sum_{x\in X}x$. Then there exists a row-finite, path-finite directed forest $T$ in $M_nE$ such that $L_R(M_nE(T))\cong e_XL_R(M_nE)e_X$. If in addition $L_R(M_nE)$ is an algebraic Cuntz-Krieger algebra, then $e_XL_R(M_nE)e_X$ is isomorphic to an algebraic Cuntz-Krieger algebra.
\end{itemize}
\end{corollary}
\begin{proof} Since $H_{SE}(X)\backslash X$ (resp. $H_{M_nE}(X)\backslash X$) is a finite set, by \cite[Lemma 3.6]{C} there is a row-finite, path-finite directed forest $T$ in $SE$ (resp. $M_nE$) which satisfies the conditions of Theorem \ref{thm3}. Thus the result follows from Theorem \ref{thm3}.
\end{proof}

An idempotent $e$ of an algebra $A$ is called full idempotent if $AeA=A$.

\begin{remark} In Corollary \ref{cor3} if in addition we assume that $E$ is a graph with finitely many vertices and $E^0\subseteq X$, then the smallest saturated subset of $(SE)^0$ (resp. $(M_nE)^0$) containing $X$ is $(SE)^0$ (resp. $(M_nE)^0$) and so $e_X$ is a full idempotent.
\end{remark}

\begin{proposition}\label{cor4} Let $R$ be a commutative unital ring and $A$ be an algebraic Cuntz-Krieger $R$-algebra. Then $\mathbb{M}_n(A)$ is isomorphic to an algebraic Cuntz-Krieger algebra for each positive integer $n$.
\end{proposition}

\begin{proof} Let $E$ be a finite graph with no sinks and no sources such that $A=L_R(E)$ and $n$ be a positive integer. $M_nE$ is a finite graph with no sinks and so by Corollary \ref{cor2}, $L_R(M_nE)$ is isomorphic to an algebraic Cuntz-Krieger algebra. Thus $\mathbb{M}_n(L_R(E))\cong L_R(M_nE)$ is isomorphic to an algebraic Cuntz-Krieger algebra.
\end{proof}

\section{Algebras that are Morita equivalent to algebraic Cuntz-Krieger algebras}

In this section we show that if a unital algebra $A$ is stably isomorphic to an algebraic Cuntz-Krieger algebra, then $A$ is isomorphic to an algebraic Cuntz-Krieger algebra. Also we show that if $A$ is Morita equivalent to an algebraic Cuntz-Krieger algebra, then $A$ is isomorphic to an algebraic Cuntz-Krieger algebra.

\begin{definition} Let $R$ be a commutative unital ring, $A$ be an $R$-algebra, $e^2=e\in \mathbb{M}_n(A)$ and $f^2=f\in \mathbb{M}_m(A)$. $e$ is called Murray-von Neumann equivalent to $f$, denoted $e\sim f$, if there exist $x\in \mathbb{M}_{m,n}(A)$ and $y\in \mathbb{M}_{n,m}(A)$ such that $e=yx$ and $f=xy$.
\end{definition}

\begin{example} Let $R$ be a commutative unital ring and $E$ be a graph. Let $v\in E^0$ be a regular vertex. Thus by the Cuntz-Krieger relations we have $v\sim\sum_{e\in s^{-1}_E(v)}r_E(v)$.
\end{example}

For an idempotent $e\in A$ and a positive integer $n$, $ne$ denotes the idempotent $M\in \mathbb{M}_n(A)$ that $M=(m_{ij})$, $m_{ii}=e$ for each $i$ and $m_{ij}=0$ for each $i\neq j$.

The proof of the following lemma is similar to the proof of \cite[Lemma 4.6]{AR}, and we give the proof for the reader's convenience.

\begin{lemma}\label{lemm1} Let $k$ be a field, $E$ be a finite graph with no sinks and no sources such that every vertex of $E$ is a base point of at least one cycle of length one, $\{v, e| v\in E^0, e\in E^1\}$ be a Cuntz-Krieger $E$-family and $f$ be a full idempotent of $\mathbb{M}_n(L_k(E))$. Then there exists a set $\{m_v|v\in E^0, m_v\geq1\}$ of positive integers such that $f\sim \sum_{v\in E^0}m_vv$.
\end{lemma}

\begin{proof} By Theorem 3.5 of \cite{AMP}, there exists a set $\{n_v|v\in E^0, n_v\geq0\}$ of non-negative integers such that $f\sim \sum_{v\in E^0}n_vv$. Let $H_0$ be the smallest hereditary subset of $E^0$ which contains $S_0=\{v|v\in E^0, n_v\neq0\}$. By Lemma 4.5 of \cite{AR}, $H_0$ is saturated. Put $g=\sum_{v\in S_0}v\in I_{H_0}$. Since $f\sim \sum_{v\in E^0}n_vv$, the ideal generated by $f$ is equal to the ideal generated by $e_{11}\otimes g$, where $\{e_{ij}\}_{i, j}$ is a system of matrix units. Thus $e_{11}\otimes g$ is a full idempotent in $\mathbb{M}_n(L_k(E))$ and so $g$ is a full idempotent of $L_k(E)$. Thus $I_{H_0}=L_k(E)$ and hence $H_0=E^0$. Thus for each $w\in E^0$, there exists $v\in S_0$ such that $v\geq w$. Put $E^0\backslash S_0=\{w_0, w_1,\cdots, w_m\}$. Let $v\in S_0$ such that $v\geq w_0$. The similar argument as in the proof of \cite[Lemma 4.3]{AR} shows that $v\sim w_0+\sum_{u\in E^0} m_u(v, w_0)u$, where $m_u(v, w_0)\geq 0$ and $m_v(v, w_0)\geq |\{e\in E^1|s_E(e)r_E(e)=v\}|\geq 1$. Hence by using such equations for all $w_0, \cdots, w_m$ we achieve $f\sim \sum_{v\in E^0} n'_vv$, where $n'_v\geq 1$ for all $v\in E^0$.
\end{proof}

\begin{proposition}\label{prop1} Let $k$ be a field, $E$ be a finite graph with no sinks and no sources, $n$ be a positive integer and $e$ be a full idempotent of $\mathbb{M}_n(L_k(E))$. Then there exists a finite graph $F$ that has no sinks and no sources such that $L_k(F)\cong e\mathbb{M}_n(L_k(E))e$.
\end{proposition}

\begin{proof} First, we assume that every vertex of $E$ is a base point of at least one cycle of length one. By Proposition 9.3 of \cite{AT} and its proof there exists an isomorphism $\Phi:\mathbb{M}_n(L_k(E))\rightarrow L_k(M_nE)$ such that for each $v\in E^0$, $K_0(\Phi)([e_{11}\otimes v])=[v]$. Let $e$ be a full idempotent of $\mathbb{M}_n(L_k(E))$, thus by Lemma \ref{lemm1}, $e\sim \sum_{v\in E^0}m_vv$ with $m_v\geq 1$ for all $v\in E^0$. Since $L_k(M_nE)$ is separative by \cite[Theorem 6.3]{ABC}, $\Phi(e)$ is Murray-von Neumann equivalent to $e_X\in L_k(M_nE)$ such that $X$ is a finite, hereditary subset of $(M_nE)^0$ with $E^0\subseteq X$. By Corollary \ref{cor3}, $e_XL_k(M_nE)e_X\cong L_k(F)$ for some finite graph $F$ with no sinks and no sources. Thus $e\mathbb{M}_n(L_k(E))e\cong \Phi(e)L_k(M_nE)\Phi(e)\cong e_XL_k(M_nE)e_X\cong L_k(F)$ and the result follows. Now let $E$ be a finite graph with no sinks and no sources. Since $M_nE$ is a finite graph with no sinks, we can apply \cite[Theorem 3.1]{N} to get a finite graph $G$ with no sinks and no sources, and every vertex of $G$ is a base point of at least one cycle of length one, such that $L_k(M_nE)$ and $L_k(G)$ are Morita equivalent. Therefore there exist positive integer $m$ and full idempotent $f$ of $\mathbb{M}_m(L_k(G))$ such that $L_k(M_nE)\cong f\mathbb{M}_m(L_k(G))f$. Therefore $e\mathbb{M}_n(L_k(E))e\cong f'\mathbb{M}_m(L_k(G))f'$ for some full idempotent $f'$ of $\mathbb{M}_m(L_k(G))$. Thus by the first case there exists a finite graph $F$ that has no sinks and no sources such that $L_k(F)\cong f'\mathbb{M}_m(L_k(G))f'\cong e\mathbb{M}_n(L_k(E))e$ and the result follows.
\end{proof}

\begin{corollary}\label{cor5} Let $k$ be a field and $A$ be a unital $k$-algebra which is Morita equivalent to an algebraic Cuntz-Krieger $k$-algebra. Then $A$ is isomorphic to an algebraic Cuntz-Krieger algebra.
\end{corollary}

\begin{proof} Let $E$ be a finite graph with no sinks and no sources such that $A$ is Morita equivalent to the algebraic Cuntz-Krieger algebra $L_k(E)$. Therefore there exists a positive integer $n$ and full idempotent $e$ of $\mathbb{M}_n(L_k(E))$, that $A\cong e\mathbb{M}_n(L_k(E))e$. Therefore by Proposition \ref{prop1}, there exists a finite graph $F$ that has no sinks and no sources such that $L_k(F)\cong e\mathbb{M}_n(L_k(E))e\cong A$.
\end{proof}

\begin{corollary}\label{cor6} Let $k$ be a field and $A$ be a unital $k$-algebra which is stably isomorphic to an algebraic Cuntz-Krieger $k$-algebra. Then $A$ is isomorphic to an algebraic Cuntz-Krieger algebra.
\end{corollary}

\begin{proof} By \cite[Proposition 9.10]{AT}, $A$ is Morita equivalent to an algebraic Cuntz-Krieger $k$-algebra. Therefore the result follows by Corollary \ref{cor5}.
\end{proof}

\begin{corollary}\label{cor7} Let $k$ be a field and $A$ be a $k$-algebra. Then the following are equivalent:
\begin{itemize}
\item[(1)] $A$ is an algebraic Cuntz-Krieger algebra;
\item[(2)] $\mathbb{M}_n(A)$ is isomorphic to an algebraic Cuntz-Krieger algebra for each $n\in \mathbb{N}$;
\item[(3)] $\mathbb{M}_n(A)$ is isomorphic to an algebraic Cuntz-Krieger algebra for some $n\in \mathbb{N}$.
\end{itemize}
\end{corollary}

\begin{proof} $(1) \Rightarrow (2)$ follows from Proposition \ref{cor4}.

$(2) \Rightarrow (3)$ is obvious.

$(3) \Rightarrow (1)$ Assume that $\mathbb{M}_n(A)$ is isomorphic to an algebraic Cuntz-Krieger algebra for some $n\in \mathbb{N}$. Thus $\mathbb{M}_n(A)$ is unital and so $A$ is a unital $k$-algebra. Since $A$ is stably isomorphic to $\mathbb{M}_n(A)$, the result follows from Corollary \ref{cor6}.
\end{proof}

\begin{corollary}\label{cor7} Let $k$ be a field, $A$ be an algebraic Cuntz-Krieger $k$-algebra and $e$ be a non-zero idempotent of $A$. Then $eAe$ is isomorphic to an algebraic Cuntz-Krieger algebra.
\end{corollary}

\begin{proof} Let $A=L_k(E)$ where $E$ is a finite graph with no sinks and no sources. Let $e$ be a non-zero idempotent of $A$ and $I$ be the ideal in $L_k(E)$ generated by $e$. As $eAe\subseteq I$ we have $eAe\subseteq eIe$ and so $eIe=eAe$. Since $I$ is generated by an idempotent $e$, by \cite[The proof of Proposition 5.2 and Theorem 5.3]{AMP} $I$ is a graded ideal of $L_k(E)$. Therefore there exists a hereditary saturated subset $H$ of $E^0$ such that $I=I_H$. Let $E_H=(H, s^{-1}_E(H), r_E, s_E)$. $E_H$ is a finite graph with no sinks and by \cite[Lemma 2.4]{APS}, $L_k(E_H)$ is Morita equivalent to $I_H$. Thus by Corollary \ref{cor2} and Corollary \ref{cor5}, $I_H$ is isomorphic to the algebraic Cuntz-Krieger algebra $B$. Hence $eAe=eI_He$ is isomorphic to the $fBf$ for some full idempotent $f$ of $B$. By Proposition \ref{prop1}, there exists a finite graph $F$ that has no sinks and no sources such that $L_k(F)\cong fBf$. Thus $eAe\cong L_k(F)$ and the result follows.
\end{proof}

\begin{corollary}\label{cor8} Let $k$ be a field, $A$ be an algebraic Cuntz-Krieger $k$-algebra and $e$ be a non-zero idempotent of $\mathbb{M}_\infty(A)$. Then $e\mathbb{M}_\infty(A)e$ is isomorphic to an algebraic Cuntz-Krieger algebra.
\end{corollary}

\begin{proof} We use the same argument as in the proof of \cite[Corollary 4.10]{AR}. Let $A=L_k(E)$ where $E$ is a finite graph with no sinks and no sources. By Theorem 3.5 of \cite{AMP}, there exists a set $\{n_v|v\in E^0, n_v\geq0\}$ of non-negative integers such that $e\sim \sum_{v\in E^0}n_vv$. Let $X=\{v\in E^0| n_v\neq 0\}$ and $f=\sum_{v\in X}v$. Thus $f$ is a non-zero idempotent of $L_k(E)$ and by Corollary \ref{cor7}, there exists a finite graph $F$ that has no sinks and no sources such that $L_k(F)\cong fL_k(E)f$. By \cite[Theorem 5.3]{AMP}, the ideal of $\mathbb{M}_\infty(A)$ generated by $e$ and the ideal of $\mathbb{M}_\infty(A)$ generated by $e_{11}\otimes f$ coincide. Thus $fAf\otimes \mathbb{M}_\infty(k)\cong (e_{11}\otimes f)\mathbb{M}_\infty(A)(e_{11}\otimes f)\otimes \mathbb{M}_\infty(k)\cong e\mathbb{M}_\infty(A)e\otimes \mathbb{M}_\infty(k)$. Therefore $e\mathbb{M}_\infty(A)e$ is stably isomorphic to an algebraic Cuntz-Krieger algebra and the result follows from Corollary \ref{cor6}.
\end{proof}
\section*{acknowledgements}
The author would like to thank to the referee for a careful reading of this paper and making many helpful suggestions that improved the presentation of the paper.

\end{document}